\documentclass{amsart}
\usepackage{amsmath, amsfonts, amsthm, amssymb,epsfig}
\newtheorem{theorem}{Theorem}[section]
\newtheorem{proposition}[theorem]{Proposition}
\newtheorem{corollary}[theorem]{Corollary}
\newtheorem{lemma}[theorem]{Lemma}

\newtheorem{definition}[theorem]{Definition}
\newtheorem{example}[theorem]{Example}
\newtheorem{xca}[theorem]{Exercise}

\theoremstyle{remark}
\newtheorem{remark}[theorem]{Remark}

\numberwithin{equation}{section}

\newcommand{\abs}[1]{\lvert#1\rvert}

\newcommand{\blankbox}[2]{%
  \parbox{\columnwidth}{\centering
    \setlength{\fboxsep}{0pt}%
    \fbox{\raisebox{0pt}[#2]{\hspace{#1}}}%
  }%
}

\begin{document}

\title{On Banach spaces with angelic Mackey duals}

\author{Douglas Mupasiri}
\address{Department of Mathematics, University of Northern Iowa , Cedar Falls, Iowa 50613}
%\curraddr{Department of Mathematics and Statistics,
%Case Western Reserve University, Cleveland, Ohio 43403}
\email{mupasiri@math.uni.edu}

%\author{Author Two}
%\address{Mathematical Research Section, School of Mathematical Sciences,
%Australian National University, Canberra ACT 2601, Australia}

%    General info

\date{August 1, 2018.}

%\dedicatory{This paper is dedicated to Prof. Joseph Diestel}


\begin{abstract}
We show that if $X$ is a sequentially reflexive Banach space, then its Mackey dual $(X^{*},\tau (X^{*}, X))$ is an angelic space. This builds on a result of J. Howard which says that in the Mackey dual $(X^{*}, \tau (X^{*}, X))$ of a Banach space $X$, relative sequential compactness is, in general, strictly stronger than relative compactness and that the two notions of compactness are equivalent if $X$ is reflexive or separable. Our main result gives a characterization of the sequentially reflexive spaces as the Banach spaces $X$ for which the the finest locally convex topology on $X^{*}$ with the same precompact sets as the Mackey topology $\tau (X^{*}, X)$ is the bound extension of $\tau (X^{*}, X)$. 
\end{abstract}

\maketitle

%\section*{Introduction}
\section{Introduction}
A result of Joe Howard's \cite{HJ} says that in the Mackey dual $(X^{*},\tau (X^{*},X))$ of  a Banach space $X$, relative sequential compactness is strictly stronger than relative compactness and the two notions of compactness are equivalent if $X$ is either reflexive or separable. More recently P. Orno \cite{OP} and M. Valdivia \cite{VM93} independently characterized the sequentially reflexive Banach spaces as precisely those Banach spaces which do not contain a subspace isomorphic to $\ell_{1}$. When combined with Emmanuele's characterization of Banach spaces which do not contain $\ell_{1}$ \cite{EG} and Grothendieck's characterization of Mackey compact subsets of $X^{*}$ \cite{GA1}, Orno's and Vildivia's result suggests that in the Mackey duals of sequentially reflexive Banach spaces the notions of (relative) sequential compactness and (relative) compactness are equivalent. Below we give a direct elementary proof of a result along similar lines. Actually, we show more. Specifically, we show that the sequentially reflexive Banach spaces are exactly the Banach spaces $X$ for which an Eberlein-\v{S}mulian type theorem holds for the Mackey dual $(X^{*},\tau (X^{*},X))$ and for which the finest locally convex topology on $X^{*}$ with the same precompact sets as $\tau (X^{*}, X)$ is a bound topology. Our main result adds a new perspective to the list of characterizations of Banach spaces not containing $\ell_{1}$.  As might be expected the proof of our main result leans heavily on Rosenthal's $\ell_{1}$ theorem \cite{RH}.

%% The correct journal style for \specialsection is all uppercase; a known bug
%% in amsart.cls prevents this, so input must be uppercase until it is fixed.
%\specialsection*{This is a Special Section Head}

%\specialsection*{PRELIMINARIES}
\section{Preliminaries}
We use the monographs by G. K\"{o}the \cite{KG} and H. Schaefer \cite{SH} as basic references for the theory of locally convex spaces. Throughout, $X$ denotes a Banach space over the field of real numbers, $B_{X}$ denotes the closed unit ball of $X$. The dual of $X$, $X^{*}$, is the space of continuous linear forms on $X$ and $B_{X^{*}}$ is the closed unit ball of $X^{*}$. The second dual of $X$, $X^{**}$, is the dual of $X^{*}$. We then have two dual pairs $<X,X^{*}>$ and $<X^{*},X^{**}>$ and the four topologies the {\it weak} topology on $X$, $\sigma (X,X^{*})$, the $weak^{*}$ topology on $X^{*}$
$\sigma (X^{*},X)$, the {\it weak} topology on $X^{*}$, $\sigma (X^{*},X^{**})$, 
and the $weak^{*}$ topology $\sigma (X^{**}, X^{*})$ on $X^{**}$.

\

If $(X,||\cdot ||)$ is a Banach space, then the dual norm on $X^{*}$ will be denoted by $||\cdot ||^{*}$ or just $||\cdot ||$ if there is no risk of misunderstanding. The topology defined by a norm $||\cdot ||$ will also be denoted by $||\cdot ||$. If $\mathcal{T}_1$ and $\mathcal{T}_2$ are topologies on $X$, then $\mathcal{T}_1 \prec\mathcal{T}_2$ means that the topology  $\mathcal{T}_1$ ie weaker than the topology $\mathcal{T}_2$. For example, $\mathcal{T}\prec ||\cdot||$ means that the topology $\mathcal{T}$ is courser than the topology defined be the norm $||\cdot||$. 

\

If $(X,\mathcal{T})$ is a locally convex space and $Y$ is a closeed subspace of $X$, then $q:X\rightarrow X/Y$ will denote the canonical quotient mapping. Objects and concepts in $X/Y$ will be labeled with a hat. For example, $\hat{\mathcal{T}}$ will denote the quotient topology of $\mathcal{T}$, and elements of $X/Y$ will be written as $\hat{x}$ 

\

We shall use the notations $(B_{X},weak)$ for the closed unit ball of $X$ with the relative weak topology. The Mackey topology, $\tau (X^{*},X)$ on $X^{*}$ is the topology of uniform convergence on all the weakly compact absolutely convex subsets of $X$. Since the closed absolutely convex hull of a weakly compact set is again weakly compact (by a result of Krein, see \cite[Ch. IV, 11.4, page]{SH}), the ``absolutely convex'' qualification is omitted in some references. By the Mackey-Arens theorem \cite[Ch. IV, 3.2, page 131]{SH}, $\tau (X^{*},X)$ is the finest locally convex topology on $X^{*}$ whose topological dual is $X$. 

\

Recent work on the Mackey dual can be found in Dowling and Mupasiri \cite{DPN} Schl\"{u}chtermann and Wheeler \cite{SG1, SG2} and in earlier works by A. Grothendieck \cite{GA2}, Joe Howard \cite{HJ}, N. Kalton \cite{KN}, and J. Webb \cite{WJ}.

\

The following dual characterization of reflexive Banach spaces is a direct consequence of the definitions:  A Banach space is reflexive if and only if the Mackey topology $\tau (X^{*}, X)$ and the norm topology agree on $X^{*}$. This fact led Jonathan Borwein to make the following definition \cite{BJ}

\begin{definition} A Banach space is {\it sequentially reflexive} if every $\tau (X^{*},X)$ convergent sequence in $X^{*}$ is norm convergent.
\end{definition}

\

Borwein then asked for a characterization of sequentially reflexive spaces. Shortly thereafter, P. Orno and M. Valdivia characterized sequentially reflexive spaces as the Banach spaces not containing $\ell_{1}$. On reflection, Orno's and Valdivia's result suggests that in the Mackey duals of sequentially reflexive Banach spaces the notions of (relative) sequential compactness and (relative) compactness are equivalent. Below we give a proof of a result along similar lines. Actually, we show more. Specifically, we show that the sequentially reflexive Banach spaces are exactly the Banach spaces $X$ for which an Eberlein-\v{S}mulian type theorem holds for the Mackey dual $(X^{*},\tau (X^{*},X))$ and for which the finest locally convex topology on $X^{*}$ with the same precompact sets as $\tau (X^{*}, X)$ is a bound topology. Our main goal in  this note is to prove this equivalence. Our secondary objective is to put the aforemntioned result in the larger context of some new, some not so well-known, and some well-known characterizations of Banach spaces not containing $\ell_{1}$ that we believe fit naturally together.

\

%We shall state our first result in a manner that parallels J. D. Pryce's %formulation of the Eberlein-$\check{S}$mulian theorem \cite{PJD71}. 
D.H. Fremlin's notion of angelic spaces will figure prominently in the sequel. We now define the notion.

\begin{definition} (see, e.g., K. Floret \cite{FK80} or J.D. Pryce \cite{PJD71}) A Hausdorff topological space $T$ is said to have {\it countably determined compactness} for subsets of $X$, or to be an {\it angelic space}, if every countably compact subset $A$ of $T$ has the following two properties: (i) $A$ is relatively compact and (ii) every $a\in\overline{A}$ is the limit of a sequence in $A$.
\end{definition}

%%%%%%%%%%%%%%%%%%%%%%%%%%%%%%%%%%%%%%%%%%%%%%%%%%%%%%%%%%%%%%%%%%%%%%%%
%\footnote{Here is an example of a footnote. Notice that this footnote
%text is running on so that it can stand as an example of how a footnote
%with separate paragraphs should be written.
%\par
%And here is the beginning of the second paragraph.}
%
%%%%%%%%%%%%%%%%%%%%%%%%%%%%%%%%%%%%%%%%%%%%%%%%%%%%%%%%%%%%%%%%%%%%%%%%

The arguments we use to prove the results in this note depend on considerations of different locally convex topologies on the dual space $X^{*}$ that are related to the Mackey topology. For a locally convex space $(E,\tau)$, always assumed Hausdorff, we follow Webb \cite{WJ} and make the following definitions.

\begin{enumerate}
\item $E^{\#}$ is the algebraic dual of $E$;
\item $E'$ is the topological dual of $E$;
\item $E^{b} = (E, \tau)^{b}$ is the space of all bounded linear functionals on $E$\\
\hspace*{4 mm} = $\{f\in E^{\#} : \sup_{x\in B}|f(x)| < \infty\,\, {\rm for\,\, all} \,\tau - {\rm bounded} \,\,B\subset E\}$; 
\item $\tau ^{+}$ is the finest locally convex topology with the same convergence sequences as $\tau$;
\item $\tau ^{p}$ is the finest locally convex topology with the same precompact sets as $\tau$;
\item $\tau ^{g}$ is the topology on $E^{b}$ of uniform convergence on all subsets $A$ of $E$ such that every sequence of points in $A$ has a Cauchy subsequence;
\item $\tau ^{n}$ is the topology on $E^{b}$ of uniform convergence on the family of sets formed by the ranges of all $\tau$-null sequences in $E$;
\item $\tau ^{0}$ is the topology on $E^{b}$ of uniform convergence on all the $\tau$ - precompact subsets of $E$.
\end{enumerate}

\

We shall also use the following notation, following Schaefer.

\begin{itemize}
\item  [(8)] $\tau _{f}$ is the finest topology on $E'$ which agrees with $\sigma (E',E)$ on each equicontinuous subset of $E'$; and
\item  [(9)] $\tau _{c}$ is the topology of $E'$ of uniform convergence on all the compact subsets of $E$.
\end{itemize}

\begin{remark} In general, $\mathcal{T} ^{+}$ and $\mathcal{T} ^{p}$ are incomparable, that is neither $\mathcal{T}^{+} \preceq \mathcal{T}^{p}$ nor  $\mathcal{T} ^{p}\preceq \mathcal{T}^{+}$ holds.
\end{remark}

\begin{proposition} If $(E,\mathcal{T})$ is locally convex, then $(E,\mathcal{T}^{p})' = E^{b}$ and $\mathcal{T}\preceq \mathcal{T}^{p} \preceq \mathcal{T}^{b}$. Also $(E,\mathcal{T}^{+})' = E^{+}$ and $\mathcal{T}\preceq\mathcal{T}^{+}\preceq \mathcal{T}^{b}$.
\end{proposition}

\begin{proof}
To prove the first assertion, let $A\subset E$ be $\mathcal{T}$-bounded. If $A$ is not $\mathcal{T}^{p}$-bounded, then there exists a $\mathcal{T}^{p}$-neighborhood $U$ of $0$ and a sequence $\{a_n\}$ in $A$ such that $a_{n}\notin nU$ for all $n\in \mathbb{N}$. But then $\{(1/\sqrt{n})a_{n}\}$ is $\mathcal{T}$-null, hence $\mathcal{T}$-compact and thus $\mathcal{T}^{p}$-precompact, contradicting the fact that the sequence $\{(1/n)x_{n}\}$ is not $\mathcal{T}^{p}$-bounded. This proves that $\mathcal{T}\preceq \mathcal{T}^{p} \preceq \mathcal{T}^{b}$ by the definition of $\mathcal{T}^{b}$.

To prove the second assertion, let $A\subset E$ be $\mathcal{T}$-bounded. If $A$ is not $\mathcal{T}^{+}$-bounded, then there exists a $\mathcal{T}^{+}$-neighborhood $U$ of $0$ and a sequence $\{a_n\}$ in $A$ such that $a_{n}\notin nU$ for all $n\in \mathbb{N}$. But then $\{(1/\sqrt{n})a_{n}\}$ is $\mathcal{T}$-null, hence $\mathcal{T}^{+}$-null, a contradiction. This proves that $\mathcal{T}\preceq\mathcal{T}^{+}\preceq \mathcal{T}^{b}$ by the definition of $\mathcal{T}^{b}$. 
(See also Proposition 2.5 and comment after Corollary 1.7 in \cite{WJ}). 
\end{proof}

The following two propositions, due to J. H. Webb, characterize the conditions under which $\mathcal{T} ^{+}$ and $\mathcal{T} ^{p}$ are comparable.

\begin{proposition} The following statements are equivalent
\begin{itemize}
\item[(a)] $\mathcal{T}^{p} \preceq \mathcal{T}^{+}$
\item[(b)] $E^{+} = E^{b}$
\item[(c)] $\mathcal{T}^{+} = (\mathcal{T}^{p})^{+}$
\item[(d)] $\mathcal{T}^{+} = (\mathcal{T}^{g})^{0}$
\item[(e)] $\mathcal{T}^{+} = (\mathcal{T}^{n})^{0}$
\item[(f)] $\mathcal{T}^{nn} \preceq \mathcal{T}^{+}$
\end{itemize}
\end{proposition}

\begin{proposition}Let $(E,\mathcal{T})$ be a Hausdorff locally convex space. Then the following conditions are equivalent:

\begin{itemize}
\item[(a)] $\mathcal{T}^{+} \preceq \mathcal{T}^{p}$
\item[(b)] Every $\mathcal{T}$-precompact subset of $E$ is $\mathcal{T}^{+}$-precompact.
\item[(c)] Every $\mathcal{T}$-limited subset of $E^{b}$ is $\mathcal{T}^{0}$-compact
\item[(d)] The topologies $\mathcal{T}^{+}$ and $\sigma (E,E^{+})$ coincide on the $\mathcal{T}$-precompact subsets of $E$.
\item[(e)] The topologies $\mathcal{T}^{0}$ and $\sigma (E^{b} ,E)$ coincide on the $\mathcal{T}$-limited subsets of $E^{b}$.
\end{itemize}
\end{proposition}

\
\begin{theorem} (The Grothendieck Completeness Theorem) Let $(E,\mathcal{T})$ be a l.c.s. and let $\mathfrak{S}$ be a saturated family of bounded sets covering $E$. Then $E'$ is complete under the $\mathfrak{S}$-topology iff every $f\in E^{\#}$ which is $\mathcal{T}$-continuous on each $S\in \mathfrak{S}$, is continuous on $(E,\mathcal{T})$
\end{theorem}

\begin{corollary} A l.c.s. $(E,\mathcal{T})$ is complete iff every $g\in (E')^{\#}$ which is $\sigma (E',E)$-continuous on every equicontinuous subset of $E'$ is $\sigma (E',E)$-continuous on all of $E'$. 
\end{corollary}

\begin{lemma}
Let $E = X^{*}$ and $\tau = \tau (X^{*}, X)$ in the preceding definitions. Then both $(X^{*},\tau ^{+})$ and ($X^{*},\tau ^{p})$ are complete locally convex spaces.
\end{lemma}

\begin{proof}
We first note that $(X^{*}, \tau ^{+})' = (X^{*})^{+}$ \cite[p. 344]{WJ} and $(X^{*},\tau ^{p})' = (X^{*})^{b} = X{**}$ \cite[p. 349]{WJ}. We apply Grothendieck's Completeness Theorem \cite[p. 149]{SH} to establish the lemma. To that end, let $f\in ((X^{*})^{+})^{\#}$ be $\sigma ((X^{*})^{+},X^{*})$-continuous on every equicontinuos subset of $(X^{*})^{+}$ (i.e., on the $\tau$-limited subsets of $(X^{*})^{+}$ by \cite[Proposition 1.3, p. 343]{WJ}). By the definition of the Mackey topology $\tau$, every $\sigma (X,X^{*})$-compact subset $X$ is $\tau$-limited. Hence $f$ is $\sigma ((X^{*})^{+},X^{*})$-continuous on each $\sigma (X,X^{*})$-compact subset of $X$. Since $\sigma (X,X^{*})$ is the subspace topology induced on $X$ by the topology $\sigma ((X^{*})^{+},X^{*})$, we deduce that $f$ is  $\sigma (X,X^{*})$-continuous on each $\sigma (X,X^{*})$-compact subset of $X$. Now the range of a norm-null sequence in $X$ is $\sigma (X,X^{*})$-compact and so $\lim_{n\rightarrow\infty} f(x_{n}) = 0$ for every norm-null sequence $\{x_{n}\}_{n=1}^{\infty}$ in $X$. Hence $f\in X^{*}$. This completes the proof that $((X^{*})^{+},\tau^{+})$ is a complete locally convex space.

\

To prove that $(X^{*},\tau ^{p})$ is complete, we also apply Grothendieck's Completeness Theorem. To that end, let $g\in ((X^{*})^{b})^{\#}$ be  $\sigma ((X^{*})^{b},X^{*})$-continuous on every equicontinuos subset of $(X^{*})^{b}$ (i.e., on the $\tau^{0}$-precompact subsets of $(X^{*})^{b}$ by \cite[Proposition 2.4, page 349]{WJ}). Since $X\subset (X^{*})^{b} =  X^{**}$ and $X$ is $\sigma (X^{**},X^{*})$ dense in $X^{**}$, it follows that the canonical bilinear form of the duality $<X^{*},(X^{*})^{b}>$ places $X^{*}$ and $X$ in duality (see \cite[Ch. IV, 1.3, page 125]{SH}). The Grothendieck Interchange Lemma \cite[Corollary 4, page 93]{GA2}, says that every $\sigma (X,X^{*})$-compact subset of $X$ is $\tau^{0}$-precompact. Hence every $\sigma (X,X^{*})$-compact subset of $X$ is an equicontinuous subset of $(X^{*})^{b}$. So $g$ is $\sigma ((X^{*})^{b}, X^{*})$-continuous on every $\sigma (X,X^{*})$-compact subset of $X$. By the same argument used in the preceding paragraph, $\lim_{n\rightarrow\infty} g(x_{n}) = 0$ for every norm-null sequence $\{x_{n}\}_{n=1}^{\infty}$ in $X$. So $g\in X^{*}$. This completes the proof that $(X^{*},\tau ^{p})$ is a complete locally convex space.
\end{proof}

\section{THE MAIN RESULT}

Our main result builds on some well known results, which we now state. We first define some concepts we will need.

\begin{definition}: A continuous linear operator $T: X\longrightarrow Y$ is said to be a Dunford-Pettis operator (or completely continuous operator) iff it maps weakly compact sets to norm compact sets. A Banach space $X$ has the Dunford-Pettis Property (DP) if every weakly compact operator $T: X\longrightarrow Y$ is a Dunford-Pettis operator. Similarly, $X$ has the reciprocal Dunford-Pettis property (RDP) iff every Dunford-Pettis operator is weakly compact. If $K$ is a compact Hausdorff space, then $C(K)$ has both DP and RDP.
\end{definition}

The following characterization of sequentially reflexive Banach spaces was proved independently by P. Orno \cite{OP} and M. Valdivia \cite{VM93}.

\begin{theorem} 
A Banach space $X$ is sequentially reflexive iff it does not contain an isomorphic copy of $\ell_{1}$
\end{theorem}

Our main result will lean heavily on G. Emmanuele's characterization of Banach spaces which do not contain an isomorphic copy $\ell_{1}$ \cite{EG}

\begin{theorem} [G. Emmanuele - 1986]
Let $X$ be a Banach space. The following conditions on $X$ are equivalent:
\begin{description}
\item[(a)] Every Dunford-Pettis operator is $T: X \rightarrow Y$ is compact.
\item[(b)] Every $\tau (X^{*},X)$-compact subset of $X^{*}$ is norm-compact.
\item[(c)] $X$ contains no isormophic copy of $\ell_{1}$.
\end{description}
\end{theorem}

We first prove a lemma.

\begin{lemma}
Let $X$ be a sequentially reflexive space and let $(X^{*},\tau (X^{*},X))$ be its Mackey dual. If $A\subset X^{*}$ is $\tau (X^{*},X)$-relatively compact, then no sequence in $A$ is equivalent to the $\ell_{1}$-basis
\end{lemma}

\begin{proof}
Assume $X$ is sequentially reflexive and that a $\tau (X^{*},X)$-relatively compact set $A\subset X^{*}$ contains a sequence $\{x_{n}^*\}$ equivalent to the canonical basis of $\ell_{1}$. Then the set $\overline{A}^{\tau (X^{*},X)}$ is $\tau (X^{*},X)$-compact. The space $X$ contains no isomorphic copy of $\ell_{1}$, so by the theorem, $\overline{A}^{\tau (X^{*},X)}$ is norm-compact. So $\{x_{n}^{*} : n\in\mathbb{N}\}$ is norm-relatively compact. This is impossible.
\end{proof}

\begin{corollary}
If $X$ does not contain an isomorphic copy of $\ell_1$ (equivalently, if $X$ is a sequentially reflexive Banach space), then the Mackey dual $(X^{*},\tau (X^{*},X))$ is an angelic space.
\end{corollary}

\begin{proof}
Let $A\subset X^{*}$ be $\tau (X^{*},X)$-relatively countably compact. Then since $(X^{*},\tau (X^{*},X))$ is complete, Eberlein's theorem gives that $A$ is $\tau (X^{*},X)$-relatively compact, hence norm-relatively compact by the theorem. The result now follows trivially.
\end{proof}

{\bf Remark} It is easy to see that if $X$ is reflexive or separable, then $(X^{*},\tau(X^{*},X)$ is angelic (In the former case, $\tau (X^{*},X) = ||\cdot ||^{*}$, and in the latter case the $weak^{*}$ topology, $\sigma (X^{*},X)$, on $X^{*}$ is submetrizable. 

\

We now consider a condition on the dual pair $<X^{*},(X^{*})^{b}>$ which implies sequential reflexivity of the Banach space X. Following Webb \cite[p. 353]{WJ}, we make the following definition.

\begin{definition} A locally convex Hausdorff space $(E,T)$ is said to be sequentially barrelled if every $\sigma(E',E)$-null sequence in $E'$ is an $T$-equicontinuous set.
\end{definition}

Since the translate of an equicontinuous set is equicontinuous we obtain an equivalent definition if we replace the phrase "$\sigma(E',E)$-null sequence in $E'$" with the phrase "$\sigma(E',E)$-convergent sequence in $E'$" in the preceding definition. Also, since the class of $T$-equicontinuous sets in $E'$ is at least as large as the class of $\sigma(E',E)$-convergent sequences, $[\sigma(E',E)]^{n} \subset T$.

\begin{proposition} If $(E,T)$ is a sequentially barrelled space, then every $T$-precompact subset of $E$ is $\sigma (E',E)$-limited. Equivalently, the topologies $\sigma (E',E)$ and $T^{0}$ (the topology on $E^{b}$ of uniform convergence on the $T$-precompact subsets of $E$) have the same convergent sequences in $E'$. If, in addition, the dual pair $<E',E>$ satisfies Sm\'{u}lian's condition (i.e., every $\sigma (E',E)$-compact subset of $E'$ is $\sigma (E',E)$-sequentially compact), then every $\sigma (E',E)$-limited subset of $E$ is $T$-precompact  .
\end{proposition}

\begin{proof} See \cite[Proposition 4.1(3), p. 353 and Proposition 4.7, p. 355]{WJ}
\end{proof}

Since  $\tau = \tau (X^{*}, X)$ is sequentially barrelled, the preceding result implies that the topologies $\tau^{0}$ and $\sigma (X,X^{*})$ have the same convergent sequences in $E$. Indeed, $\tau^{0}|_{X}$ is the finest locally convex topology on $X$ with the same convergent sequences as $\sigma (X,X^{*})$, i.e., $\tau^{0}|_{X}= \sigma (X,X^{*})^{+}$ \cite[Proposition 2, p. 74]{KN} or \cite[Proposition 1.1 and Proposition 1.3]{WJ}. This naturally leads one to ask what implications, if any, imposing the stronger condition that the topologies $\tau^{0}$ and $\sigma ((X^{*})^{b},X^{*})$ agree on bounded subsets of $(X^{*})^{b}$ would have for $X$. We address this in the next theorem. While this result in known, we believe the locally convex space theory perspective it offers is new.

\

\begin{theorem} Suppose $\tau = \tau (X^{*}, X)$ and $\tau^{0}$ is the topology on $X^{**}$ of uniform convergence on the precompact subsets of $\tau$. If $\tau^{0}$ agrees with $\sigma (X^{**},X^{*})$ on the bounded subsets of $X^{**}$, then $X$ does not contain a subspace isomorphic to $\ell_{1}$. 
\end{theorem}

\begin{proof} Let $P$ be a precompact subset of $(X^{*},\tau ((X^{*},X))$. Then for every $\alpha > 0, \,\, P^{o}\cap \alpha B_{X^{**}}$ contains the intersection $V\cap \alpha B_{X^{**}}$, where $P^{0}$ is the absolute polar of $P$ and $V = V(x_{1}^{*}, \dots ,x_{n}^{*};\delta) := \{x^{**}\in X^{**}: |x^{**}(x_{i})| < \delta,\,\, i=1\dots n\}$ for some $\delta > 0$, is a $\sigma (X^{**},X^{*})$ - neighborhood of $0$. It now follows that $\overline{aco}(P^{oo}\cup (\alpha B_{X^{**}})^{o}) \subset \overline{aco}(V^{o}\cup (\alpha B_{X^{**}})^{o})\subset \overline{aco}(V^{o} + (\alpha B_{X^{**}})^{o}) \subset \overline{aco}(V^{o} + (\frac{1}{\alpha} B_{X^{*}})) \subset \overline{aco}([V^{o}] + \frac{1}{\alpha} B_{X^{*}})$, where $\overline{aco}(A)$ is the absolutely closed convex hull of $A\subset X^{*}$ and $[V^{o}]$ is linear span of $[V^{o}]$. Hence $P\subset [V^{o}] + \frac{1}{\alpha} B_{X^{*}}$ for all $\alpha > 0$. So $P\subset \bigcap_{\alpha > 0}([V^{o}] + \frac{1}{\alpha} B_{X^{*}}) = [V^{o}]$. Since $[V^{o}]$ is a finite dimensional subspace of $X^{*}$ and $P$ is bounded, we get that $P$ is $||\cdot ||$-relatively compact. So $X$ does not contain a subspace isomorphic to $\ell_{1}$, by Theorem 3.3.
\end{proof}

\begin{corollary} Let $\tau = \tau (X^{*}, X)$ and $\tau^{0}$ be as in the preceding theorem and suppose that $\tau^{0}$ agrees with $\sigma (X^{**},X^{*})$ on the bounded subsets of $X^{**}$, then $\tau^{0} = ||\cdot ||_c = ||\cdot||^n$, where $||\cdot||^n$ is the topology of uniform convergence on all the norm null sequences in $X^*$.
\end{corollary}

\begin{proof}
Since every norm compact subset of $X^*$ is $\tau^{0}$ - compact, it follows it follows that $||\cdot ||_c$ is courser than $\tau^{0}$. But $||\cdot ||_c$ is the finest topology on $X^{**}$ that agrees with $\sigma (X^{**},X^{*})$ on the bounded subsets of $X^{**}$, and so $\tau^{0}$ is courser than $||\cdot ||_c$. Hence $\tau^{0} = ||\cdot ||_c$ That the two topologies coincide with $||\cdot||^n$ on $X^{**}$ follows from the Grothendieck compactness principle (which can be deduced from the Banach-Dieudonn$\acute{e}$ theorem) [12, section 21.10(3)].
\end{proof}

We conclude by proving a characterization of sequentially reflexive Banach spaces which we believe adds a new perspective and places the content of this paper in a larger context.

\

\begin{definition} Let $(E,T)$ be a locally convex space. Let $\mathcal{U}$ be the class of circled, convex subsets of $E$ such that each set in $\mathcal{U}$ absorbs every $T$-bounded subset of $E$. By \cite[Theorem 19.1, pages 181-182]{KJ}, $\mathcal{U}$ is a local base at 0 for a topology on $E$ called the {\bf bound extension of} $T$ or the {\bf bound topology derived from} $T$ and is denoted by $T^{b}$. In general,  $T\subset T^{b}$. If $T = T^{b}$, then $T$ is called a bound topology and  $(E,T)$ is called a bound space or a  bornological space. Thus $T$ is a bound topology if and only if every circled, convex subset $A\subset E$ which absorbs every $T$-bounded subset of $E$ is a $T$-neighborhood of 0. Equivalently, a bornological space is a locally convex space on which each semi-norm that is bounded on bounded sets is continuous \cite[p. 61]{SH}.
\end{definition}

The relevant facts about the bound extension of a locally convex topology are contained in the following theorem.

\begin{theorem} Let $T^{b}$ be the bound extension of a locally convex topology $T$ on a linear space $E$. Then:
\begin{itemize}
\item[(1)] $T^{b}$ is a bound topology and the class $\mathcal{B}$ of $T$-bounded subsets of $E$ is identical with the class of $T^{b}$-bounded subsets of $E$; moreover $T^{b}$ is the finest topology with the same bounded sets as $T$
\item[(ii)] a family $F$ of linear functions on $E$ to a locally convex space $H$ is $T^{b}$-equicontuous if and only if $F$ is bounded relative to the topology $T_{\mathcal{B}}$ of uniform convergence on the bounded subsets of $E$.
\end{itemize}
 
In particular, a linear function from $E$ to $H$ is $T^{b}$-continuous if and only if it is bounded. Equivalently, a linear function $f:E\rightarrow H$ is continuous if and only if $\{f(x_{n})\}$ is a null sequence in $H$ for every null sequence $\{x_{n}\}$ in $E$.
\end{theorem}
\begin{proof} See \cite[Theorem 19.3, p. 183]{KJ} and \cite[Ch. II, Theorem 8.3, p. 62]{SH}
\end{proof}

{\bf Remark} Theorem 3.3 allows us to deduce the following result.

\begin{corollary}
If a Banach space $X$ does not contain an isomorphic copy of $\ell_{1}$, then topology $\tau (X^{*}, X) := \tau\prec\tau^{+}=\tau^{p}=\tau^{b}\,\, (= ||\cdot||^{*}$). This means that $\tau^{p}$ is the bound extension of $\tau = \tau (X^{*}, X)$ in this case and hence  $\tau^{p} = \tau (X^{*}, X^{**}) = ||\cdot ||_{X^{*}}$ and consequently $(X^{*},\tau^{p})$ is a bornological space.
\end{corollary}

\begin{proof}
Obviously, $\tau^{b}$ is the norm topology $||\cdot ||^{*}$ on $X^{*}$ ($\tau$ and $||\cdot ||^{*}$ have the same bounded sets and $(X^{*}, ||\cdot ||^{*})$ is bornological).  Moreover, if $A\subset X^{*}$ is $\tau$-precompact, then it is $\tau$-relatively compact, hence $||\cdot ||^{*}$-compact, i.e., $||\cdot ||^{*}$-precompact. We have now shown that every $\tau$-precompact subset of $X^{*}$ is $\tau^{b}$-precompact and so $\tau^{b}\prec \tau^{p}$ by definition of $\tau^{p}$. Since $\tau^{p}\prec \tau^{b}$ (by Proposition 2.4) we get that $\tau^{p} =  \tau^{b}$.
\end{proof}

We note that $\tau \neq \tau^{b}$ iff $X$ is not reflexive. This is the case, in particular, for $X = c_{0}$.

\

We consider yet another example.

\

{\bf Example}: Let $E = \ell_{1}, \,\, \tau =\tau (\ell_{1}, c_{0})$. Then $\tau^{b} = \tau (\ell_{1},\ell_{\infty})$, the norm topology of $\ell_{1}$. Since $c_{0}$ contains no isomorphic copy of $\ell_{1}$, every $\tau$-compact subset of $\ell_{1}$ is norm-compact. Also, since $\ell_{1}$ is also a Schur space, every $\sigma (\ell_{1}, \ell_{\infty})$-compact set is norm-compact. So the topologies $\sigma (\ell_{1}, \ell_{\infty})$, $\tau (\ell_{1}, c_{0})$, and $\tau (\ell_{1},\ell_{\infty}) = ||\cdot ||_{\ell_{1}} = \tau^{b}$ have the same compact sets and, in fact, coincide on each of these sets. Thus, $\tau\neq \tau^{+} =\tau^{p} = \tau^{b} = ||\cdot ||_{\ell_{1}}$ in this case.

\

\begin{definition} Let $X$ and $Y$ be Banach spaces. A linear operator $T:X\rightarrow Y$ is said to be compact (respectively, weakly compact) if $T(B_{X})$ is relatively compact (respectively, relatively weakly compact). A bounded linear operator is said to be a Dunford-Pettis operator (or completely continuous) if it maps weakly compact sets to norm compact sets. We say $X$ has the Dunford-Pettis Property (DP) if every weakly compact operator $T:X\rightarrow Y$ is a Dunford-Pettis operator. Similarly, $X$ is said to have teh reciprocal Dunford-Pettis property (RDP) if every Dunford-Pettis operator $T:X\rightarrow Y$ is weakly compact. $X$ has the hereditary RDP if every closed subspace of $X$ has RDP. We note that if $K$ is a compact Hausdorff space, then $X = C(K)$ has both DP and RDP. 
\end{definition}

\

If $X$ is a Banach space and $\tau = \tau (X^{*},X)$, the relations $\tau\prec\tau^{+} =\tau^{p} = \tau^{b} = ||\cdot ||^{*}$ hold if and only if $X$ contains no isomorphic copy of $\ell_{1}$, by Corollary 3.12. It is well known $X$ has the hereditary reciprocal Dunford-Pettis property (RDP) if and only if $X$ contains no copy of $\ell_{1}$ \cite[Proposition 2.4]{SG2}. We will now show that $X$ has the RDP if $\tau^{+} = \tau^{p}$. Thus the gap between of $(X^{*}, \tau (X^{*},X))$ satisfying the condition $\tau^{+ }= \tau^{p}$ on the one hand and sequential reflexivity of $X$ on the other hand is, in a sense, the gap between $X$ having the RDP and  $X$ having the hereditary RDP. Analogously the difference between a Banach space $X$ having the RDP and $X$ having the hereditary RDP is apparently the difference between $(X^{*},\tau (X^{*},X))$ satisfying the condition $\tau^{+} = \tau^ {p}$ on the one hand and the topology $\tau^{p}$ being the bound extension of $\tau (X^{*},X)$ so that $\tau^{p} = \tau^{b}$, on the other hand. We will now establish these remarks.

\

We will use the following characterization of Banach spaces with the RDP.

\begin{theorem} If $X$ is a Banach space, then the following statements are equivalent.
\begin{itemize}
\item[(1)] $X$ has the RDP.
\item[(2)] For each Banach space $Y$, each bounded linear operator $T: X\rightarrow Y$ which maps weakly compact sets onto norm-compact sets is a weakly compact operator ($T(B_{X}$ is norm-relatively compact).
\item[(3)] For each Banach space $Y$, each bounded linear operator which maps $\sigma (X,X^{*})$-null sequences onto norm-null sequences is a weakly compact operator.
\item[(4)] If $K\subset X^{*}$ satisfies the condition\\

(*)\hspace{0.1in} $\lim_{n\rightarrow\infty}\sup_{x^{*}\in K} x^{*}(x_{n}) = 0$  for each weakly null sequence$ \{x_{n}\}$  in $X$, 

then $K$ is $\sigma (X^{*},X^{**})$-relatively sequentially compact.
\end{itemize}
\end{theorem}

\begin{proof} See \cite[Theorem 2.1, p. 6]{LT}
\end{proof}

\begin{theorem} If $X$ is a Banach space for which $(X^{*},\tau (X^{*},X))$ is angelic, then $\tau^{+ }\preceq\tau^{p}$.
\end{theorem}

\begin{proof} Let $A\subset X^{*}$ be a $\tau$-precompact set. Since $\tau$ is complete, $A$ is $\tau$-relatively compact. By hypothesis, $A$ is $\tau$-relatively sequentially compact, and hence $\tau^{+}$-relatively sequentially compact. So $A$ is $\tau^{+}$-relatively countably compact. Since $\tau^{+}$ is complete (by the Lemma 2.9), $A$ is $\tau^{+}$-relatively compact, by Eberlein's Theorem, and hence $A$ is  $\tau^{+}$-precompact. Thus $\tau^{+}$ has the same precompact sets as $\tau$, and so $\tau^{+} \preceq \tau^{p}$, by definition of $\tau^{p}$.
\end{proof}

We now give a necessary and sufficient condition for the equation $\tau^{+} = \tau^{p}$ to hold.

\begin{theorem} Let $X$ be a Banach space. Then  $\tau^{+} = \tau^{p}$ if and only if $X$ has the RDP. 
\end{theorem}

\begin{proof} $\Rightarrow$ Assume that $\tau^{+ } = \tau^{p}$. Then by Propositions 2.8 and 2.10 in \cite{WJ}, $(X^{*},\tau)^{+} =X^{**}$ and $\tau^{+ } = \sigma (X^{*},X^{**})$ on every $\tau$-precompact subset of $X^{*}$. If $A\subset X^{*}$ is $\tau$-compact, it is $\tau^{p}$-precompact and hence $\tau^{+}$-precompact (by hypothesis). Since $A$ is $\tau$-closed and $\tau\preceq \tau^{+}$, $A$ is $\tau^{+}$-closed. So $A$ is $\tau^{+}$-compact, and hence $\sigma (X^{*},X^{**})$-compact.Thus $X$ has RDP by Proposition 2.1 in \cite{SG2}. %Thus $X$ has the reciprocal Dunford-Pettis property (RDP) by \cite[p. 152]{GA1} . 

\

$\Leftarrow$ Assume that $X$ has $RDP$. Let $A\subset X^{*}$ be $\tau$-precompact. Then by Proposition 2.1 in \cite{SG2}, $\overline{A}^{\tau}$ is weakly compact. Let $\tau_{0}^{+}$ be the restriction to $\overline{A}^{\tau}$ be  of the $\tau^{+}$-topology. Let $B\subset \overline{A}^{\tau}$ be $\tau_{0}^{+}$-closed in $\overline{A}^{\tau}$ and let $b^{*}\in\overline{B}^{w}\,\,(\subset \overline{A}^{\tau})$. Since $B$ is weakly compact, it follows from the angelicity of the weak topology of a Banach space that there is a sequence $\{b_{n}^{*}\}$ in $B$ that converges weakly to $b^*$. $\{b_{n}^{*}\}$ also converges for the weak$^*$-topology to $b^*$. Since $\overline{A}^{\tau}$ is $\tau$-compact, $\{b_{n}^{*}\}$ converges for $\tau$ to $b^*$. Thus $\{b_{n}^{*}\}$ converges to $b^*$ in the $\tau^{+}$ topology. Since $B$ is $\tau^{+}$-closed in $A$, we get that $b^{*}\in B$ and so $B$ is weakly closed. Thus $\tau^{+}$ and the weak topology agree on $\overline{A}^{\tau}$, and hence $\tau^{+} \preceq \tau^{p}$ by Proposition 2.10 in \cite{WJ}.

\

It remains to show that $\tau^{p} \preceq \tau^{+}$. By Proposition 2.8 in \cite{WJ}, it suffices to show that $X^{**} =  (X^{*}, \tau)^{+}$. It is almost trivial to prove that every $f: (X^{*},\tau) \rightarrow \mathbb{R}$ which is sequentially continuous belongs to $X^{**}$ (It is bounded on bounded sets in $(X^*,||\cdot ||)$). Conversely if $f\in X^{**}$, and $\{x_{n}^{*}\}$ is $\tau$-null, the set $\{x_{n}^{*} : n\in\mathbb{N}\}\cup \{0\}$ is $\tau$-compact, and hence weakly compact, by Proposition 2.1 in \cite{SG2}. Since $\{x_{n}^{*}\}$ is weak$^{*}$-null, it is weakly-null, so $f(x_{n}^{*}) \rightarrow 0$, and so $f\in (X^{*},\tau)^{+}$.
\end{proof}

We remark that a Banach space $X$ with the RDP can contain an isomorphic copy
of $\ell_{1}$. For example, if $X = C(K)$, the space of continuous functions on a compact Hausdorff space $K$, then $X$ has the RDP and, in the case when $K$ is the closed interval $[0,1]\subset \mathbb{R}$, $X$ is separable and hence $\tau (X^{*}, X)$ is angelic (see the remark that follows). Yet $X$ contains a copy of $\ell_{1}$.

\begin{remark}
Corollary 3.5 can be formulated by saying that {\it If a Banach space $X$ does not contain an isomorphic copy of $\ell_1$, then $(X^{*},\tau (X^{*}, X)$ is angelic.} In view of Corollary 2.4 in \cite{SG2}, the following is an equivalent formulation: {\it If a Banach space X has HRDP, then $(X^{*},\tau (X^{*}, X)$ is angelic.} Actually a more general result, which we now prove, holds. 
\end{remark}

\begin{theorem} If a Banach space $X$ has RDP, then $(X^{*},\tau(X^{*},X))$ is angelic.
\end{theorem} 

\begin{proof} 
By Theorem 3.15, we have $\tau^{+} = \tau^{p}$ and so by Proposition 2.5 in \cite{WJ}, $(X^{*},\tau^{+})' = X^{**}$. Thus $\sigma (X^{*},X^{**}) \prec\tau^{+}$. If $A\subset X^{*}$ is $\tau$-relatively countably compact, then by the Eberlein theorem, $A$ is $\tau$-relatively compact, and hence $\overline{A}^{\tau}$ is $\tau$-compact. Proposition 2.1 in \cite{SG2} now  gives that $\overline{A}^{\tau}$ is $\sigma (X^{*},X^{**})$-compact. By Proposition 2.10 in \cite{WJ} the topologies $\sigma(X^{*},X^{**})$ and $\tau^{+}$ coincide on $\overline{A}^{\tau}$. This shows that $\overline{A}^{\tau}$ is $\tau^{+}$-compact, and so the topologies $\sigma (X^{*},X^{**}),\,\,\tau$ and $\tau^{+}$ coincide on $\overline{A}^{\tau}$, Since $\sigma (X^{*},X^{**})$ is angelic, given $a^{*}\in\overline{A}^{\tau}$, there is a sequence in $A$ that converges with respect to $\sigma (X^{*},X^{**})$ to $a^{*}$. This sequence perforce converges with respect to $\tau$ to $a^{*}$. Hence $(X^{*},\tau(X^{*},X))$ is angelic.
\end{proof}

\

As promised, we conclude this note by adding our results to the list of some well-known and some not so well-known characterizations of Banach spaces which do not contain a copy of $\ell_{1}$.

\

We first note two facts: [(a)] A Banach space is sequentially reflexive iff $\tau^{+} = ||\cdot ||$. [(b)] Let $\mathcal{K}$ (respectively, $\mathcal{P}$) be the family of all $||\cdot ||$-compact (respectively, $\tau$-compact) subsets of $X^*$. Note that $\mathcal{K}$ and $\mathcal{P}$ are both saturated families of bounded subsets of $X^*$ for the dual pair $<X,X^*>$ in the sense of \S\,21 in \cite{KG}. Let $||\cdot||_{c}$ be the topology on $X$ of uniform topology on the sets in $\mathcal{K}$. Then as mentioned in Corollary 3.9 $||\cdot ||^{n} = ||\cdot ||_{c}$ by Grothendieck's compactness principle. Observe too that 
$$||\cdot ||^{n} \preceq \tau^{n}\preceq \tau^{0}$$.

We now define two concepts which are mentioned in the next theorem.

\begin{definition} A Hausdorff space $X = (X,\mathcal{T})$ is a {\it k-space} iff a set $K\subset X$ is closed iff $K\cap C$ is closed for each compact $C\subset X$. Equivalently, $X$ is a {\it k-space} iff the topology of $\mathcal{T}$ is the finest topology yielding the same compact sets as itself, i.e., iff whenever  $\mathcal{T_{0}}$ is another topology on $X$ with the same compact sets as $\mathcal{T}$, then $\mathcal{T_{0}}\prec\mathcal{T}$.
\end{definition} 

\begin{definition}
A subset $C$ of Hausdorf space $X$ is said to be {\it sequentially closed} iff whenever $\{x_n\}$ is a sequence in $C$ converging to a point $x_0$, we have $x_0\in C$.The space $X$ is said to be sequential if every sequentially closed subset of $X$ is closed.  
\end{definition}

\begin{theorem}
Let $(X,||\cdot ||)$ be a Banach space, and let $\tau = \tau (X^*,X)$ be the Mackey dual topology on $X^*$. Then the following statements are equivalent:

\begin{itemize}
\item[(i)] $X$ does not contain an isomorphic copy of $\ell_1$.
\item[(ii)] $X$ is sequentially reflexive, i.e., $\tau^{+} = ||\cdot ||$ on $X^*$.
\item[(iii)] Every $\tau$-compact subset of $X^*$ is $||\cdot ||$-compact.
\item[(iv)] $X$ has HRDP, i.e., every closed subspace of $X$ has RDP.
\item[(v)] $(B_{X},\sigma (X,X^*))$ is a Fr$\acute{e}$chet-Urysohn space.
\item[(vi)] $(B_{X},\sigma (X,X^*))$ is a sequential space.
\item[(vii)] $(B_{X},\sigma (X,X^*))$ is a k-space.
\item[(viii)] For every subset $K$ of $X^*$ that is both weakly compact and $\tau$-compact, there is a sequence $\{x_{n}^{*}\}$ in $X^*$ that is both weakly null and $\tau$-null such that $K$ is contained in the closed convex hull of $\{x_{n}^{*}: n\in\mathbb{N}\}$.
\item[(ix)] $\tau^{0} = ||\cdot ||_{c}$ on $X$.
\item[(x)] $\tau^{0}|_{B_{X}} = \sigma (X,X^*)|_{B_{X}}$.
\end{itemize}
\end{theorem}

\begin{proof}
(i)$\Leftrightarrow$ (ii) was independently proved by P. Orno and M. Valdivia.\\
(i)$\Leftrightarrow$ (iii) is due to G. Emmanuele. \\
(i)$\Leftrightarrow$ (iv) is due to G. Schl$\ddot{u}$chtermann \& R. F. Wheeler.\\
(i)$\Leftrightarrow$ (v)$\Leftrightarrow$ (vi)$\Leftrightarrow$ (vii) is due to G. Schl$\ddot{u}$chtermann \& R. F. Wheeler. \\
(iii)$\Leftrightarrow$ (ix) If (iii) holds, then $\tau^{0} \preceq ||\cdot ||_{c}$. By (i) we get (ix). Concersely, if (ix) holds, then (iii) follows from \S\,21.1(4) in \cite{KG}\\
(iii)$\Leftrightarrow$ (x) is in [K$\ddot{o}$the, 21.7(1)] for the dual pair $<X,X^*>$ and families $\mathcal{B}$ of all bounded subsets of $X$ and $\mathcal{P}$ of all $\tau$-precompact subsets of $X^*$.\\
(i)$\Leftrightarrow$ (viii) is due to P. N. Dowling and D. Mupasiri
\end{proof}

{\it Remark} A sketch of the proof of (5) $\Rightarrow$ (6) in the Theorem 3.5 was shown to the author by Prof. Robert F. Wheeler. The implication can also be deduced from a result proved by N. Kalton using a completely different argument.

\

The author wishes to thank the referee for pointing out an error in an earlier version of this paper. Thanks to the referee's comment, the author was able to make significant improvements to the final version of the paper.

\bibliographystyle{amsplain}

\begin{thebibliography}{10}

\bibitem {BJ} J. Borwein, \textit{Asplund spcaes are ``sequentially reflexive",}  (preprint).

\bibitem {DJ} J. Diestel, \textit{Sequences and Series in Banach spaces}, Springer-Verlag, Graduate Texts in Mathematics 92, (1984).

\bibitem {DPN} P. N. Dowling, D. Mupasiri, \textit{A Grothendieck compactness principle for the Mackey dual topology}, J. Math. Anal. Appl. \textbf{410} (2014) 483-486.

\bibitem{EG} G. Emmanuele, \textit{A dual characterization of Banach spaces not containing $\ell_{1}$}, Bull. Pol. Acad. Sci. Math. \textbf{34}, no. 3-4, (1986), 155-160.

\bibitem {ER} R. Engelking, \textit{General Topology} Revised and completed ed., Sigma Series in Pure Mathematics, \textbf{vol. 6}, Heldermann Verlag Berlin (1989)

\bibitem {FK80} K. Floret, \textit{Weakly Compact Sets}, LNM \textbf{801}, Springer-Verlag, 1980.

\bibitem {HJ} J. Howard, \textit{Mackey compactness in Banach spaces}, Proc. Amer. Math. Soc. \textbf{37}, no. 2 (1973), 108-110.

\bibitem{GA1} A Grothendieck, \textit{Sur les applications lin\'{e}aires faiblement compactes d'espaces du type C(K)}, Canad. J. Math. \textbf{5} (1953), 129-173.

\bibitem {GA2} A. Grothendieck, \textit{Topological Vector Spaces}, Gordon and Breach Science Publishers, New York (1973).

\bibitem {KJ} J. L. Kelly, I. Namioka, and co-authors, \textit{Linear Topological Spaces}, Springer-Verlag, Graduate Texts in Mathematics 36, Reprint of the ed. published by Von Nostrand, Princeton, N.J., in series: The University series in higher
mathematics (1976).

\bibitem {KN} N. J. Kalton, \textit{Mackey duals and almost shrinking bases}, Proc. Cambridge Philos. Soc. \textbf{74} (1973), 73-81. 

\bibitem {KG} G. K\"{o}the, \textit{Topological Vector Spaces, I}, Springer-Verlag, New York, Inc. (1969).

\bibitem {LT} T. L. Leavelle, \textit{The Reciprocal Dunford-Pettis and Radon-Nikodym Properties in Banach spaces}, dissertation, August, 1984; Denton, Texas.
(digital.library.unt.edu/ark:/67531/metadc331441/: accessed October 28, 2018),
Unviversity of North Texas Libraries, Digital Library, digital.library.unt.edu.

\bibitem {OP} P. Orno, \textit{On J. Borwein's concept of sequentially reflexive Banach spaces}, arXiv:math/9201233v1 [math.FA] 9 Oct. 1991.

\bibitem {PJD71} J. D. Pryce, \textit{A device of R. J. Whitley's applied to pointwise compactness in spaces of continuous functions}, Proc. London Math. soc. (3) \textbf{23} (1971), 532-546.

\bibitem {SG1} G. Schl\"{u}chtermann, R. F. Wheeler, \textit{On strongly WCG Banach spaces}, Math. Z. \textbf{199} (1988), no. 3, 387-398.

\bibitem {SG2} G. Schl\"{u}chtermann, R. F. Wheeler, \textit{The Mackey dual of a Banach space}, Note Mat. \textbf{11} (1991), 273-287.

\bibitem {SH} H. H. Schaefer, \textit{Topological Vector Spaces}, Springer-Verlag, New York (Fifth printing, 1986).

\bibitem {RH} H. P. Rosenthal, \textit{A characterization of Banach spaces containing $\ell_{1}$}, Proc. Nat. Acad. Sci. \textbf{71}, (1974), 2411-2413.

\bibitem {VM93} M. Valdivia, \textit{Fr\'{e}chet spaces with no subspaces isomorphic to $\ell_1$}, Math. Japonica \text{38} (1993) 397-411.

\bibitem {WJ} J. H. Webb, \textit{Sequential convergence in locally convex spaces}, Proc. Cambridge Philos. Soc. \textbf{64} (1968), 341-364.

\bibitem {WA} A. Wilansky, \textit{Modern Methods in Topological Vector Spaces}, McGraw-Hill International Book Co. (1978).

\bibitem {WS} S. Willard, \textit{General Topology} Addison-Wesley Series in Mathematics, Addison-Wesley Publishing Co., Inc., Reading Massachusetts (1970).


\end{thebibliography}

\end{document}